\documentclass[12pt]{article}
\usepackage{amssymb}
\usepackage{amsmath}
\usepackage{amsthm}
\usepackage{comment}
\usepackage[showonlyrefs]{mathtools}
\mathtoolsset{showonlyrefs}
\newtheorem{proposition}{Proposition}

\newtheorem*{theorema}{Theorem}

\newtheorem{lemma}{Lemma}
\theoremstyle{remark}
\newtheorem*{remack}{Remarks and Acknowledgments}
\newtheorem*{remark}{Remark}
\def\re{\operatorname{Re}}

\def\N{\mathbb{N}}
\def\R{\mathbb{R}}
\def\C{\mathbb{C}}

\def\Z{\mathbb{Z}}
\numberwithin{equation}{section}
\title{On postsingularly finite exponential maps}
\author{Walter Bergweiler}
\date{}
\begin{document}
\maketitle
\begin{abstract}
We consider parameters $\lambda$ for which $0$ is preperiodic under the 
map $z\mapsto\lambda e^z$. Given $k$ and $l$, let $n(r)$ be the 
number of $\lambda$ satisfying $0<|\lambda|\leq r$ such that 
$0$ is mapped after $k$ iterations to a periodic point of period $l$.
We determine the asymptotic behavior of $n(r)$ as $r$ tends to $\infty$.
\end{abstract}

\section{Introduction and main result} \label{sec1}
Let $E_\lambda(z)=\lambda e^z$ where $\lambda\in\C\backslash\{0\}$.
We are interested in parameters $\lambda$ for which $0$ is preperiodic.
Note that $0$ is the only singularity of the inverse function of~$E_\lambda$.
Functions for which all singularities of the inverse are preperiodic are 
called {\em postsingularly finite}.
The term \emph{Misiurewicz map} is also used for such functions.
We do not discuss their role in complex dynamics here, but refer
to \cite{Benini2011,Devaney1997,Devaney2005,Hubbard2009,Jarque2011,Laubner,SZ} as a
sample of papers dealing with postsingularly finite exponential maps.

For $k,l\in\N$ we thus consider parameters $\lambda$ such that 
\begin{equation}\label{b1}
E_\lambda^k(0)=E_\lambda^{k+l}(0)
\end{equation}
while
\begin{equation}\label{b2}
E_\lambda^i(0)\neq E_\lambda^j(0)
\quad\text{for}\ 0<i<j<k+l.
\end{equation}
We denote by $n(r)$ the number of all $\lambda$ contained in 
$\{z\colon 0<|z|\leq r\}$ which satisfy~\eqref{b1} and~\eqref{b2}.
If $k=l=1$, then the set of all $\lambda\neq 0$ satisfying~\eqref{b1} and~\eqref{b2} is equal to
$\{2\pi i m\colon m\in\Z\backslash\{0\}\}$.
Thus $n(r)\sim r/\pi$ as $r\to\infty$.

For $m\in\N$ we put $f_m(z)=E_z^m(0)$. Thus $f_1(z)=z$ and $f_{m+1}(z)=ze^{f_m(z)}$.
\begin{theorema}
Let $k,l$ and $n(r)$ be as above. If $k+l\geq 3$, then
\begin{equation}\label{b3}
n(r)\sim \frac{1}{\sqrt{2\pi^3}} f_{k+l-1}(r)\sqrt{f_{k+l-2}(r)}
\quad\text{as}\ r\to\infty.
\end{equation}
\end{theorema}
The theorem will be proved using Nevanlinna theory.
We refer to~\cite{GO,Hayman1964} for the terminology and basic results of 
this theory. In particular, $T(r,f)$ denotes the Nevanlinna characteristic of a
meromorphic function $f$.

Nevanlinna theory makes it natural to consider 
\begin{equation}\label{b4}
N(r)=\int_0^r \frac{n(t)}{t} dt
\end{equation}
besides $n(r)$.

The theorem will be a consequence of the following two propositions.
\begin{proposition}\label{prop1}
Let $k,l$ and $N(r)$ be as above. Then
there exists a subset $E$ of $(0,\infty)$ which has finite measure such that
\begin{equation}\label{b5}
N(r)\sim T(r,f_{k+l})
\quad\text{as}\ r\to\infty,\ r\notin E.
\end{equation}
\end{proposition}
We note that this proposition suffices to show that $n(r)\to\infty$ as $r\to \infty$.
This means that given $k,l\in\N$ there exists infinitely parameters $\lambda$ such 
that~\eqref{b1} and~\eqref{b2}
hold.
\begin{proposition}\label{prop2}
Let $m\geq 3$. Then
\begin{equation}\label{b6}
T(r,f_m)\sim \frac{1}{\sqrt{2\pi^3}} 
\frac{f_{m-1}(r)}{\sqrt{f_{m-2}(r)}\prod_{j=1}^{m-3}f_j(r)}.
\end{equation}
\end{proposition}
These propositions will be proved in section~\ref{sec2} and~\ref{sec3}, before we show in 
section~\ref{sec4} how the above theorem follows from them.
\begin{remack}
The results of this paper (except for Proposition~\ref{prop3} below)
were presented in two talks in John H.\ Hubbard's seminar 
at Cornell University in the fall of 1988.
They were inspired by a talk by Ben Bielefeld in this seminar about the computation of 
Misiurewicz parameters using the spider algorithm~\cite{Bielefeld1992,Hubbard1994}.
The proof of Proposition~\ref{prop2} given below is a simplified version of the one presented in 
the seminar.

John Hubbard's seminar was my first encounter with complex dynamics.
(The purpose of my stay at Cornell University was
to visit Wolfgang H.~J.\ Fuchs~\cite{Fuchs}, a leading figure in Nevanlinna theory.)
I would like to
take this opportunity -- albeit very belatedly -- to thank John Hubbard and the 
participants of his seminar for igniting my interest in complex dynamics and for helpful discussions.
I~thank Dierk Schleicher for encouraging me to make the results of my talks in this seminar available
-- and I also thank him and Saikat Batabyal for useful comments on this manuscript.
Finally, I remain grateful to the Alexander von Humboldt Foundation for making my stay at Cornell University 
possible by granting me a Feodor Lynen research fellowship.
\end{remack}

\section{Proof of Proposition~\ref{prop1}}  \label{sec2}
For a meromorphic function $f$ and $a\in\C$  or -- more generally -- a meromorphic function $a$ 
satisfying $T(r,a)=o(T(r,f))$, a so-called \emph{small} function,
we denote by $\overline{n}(r,a,f)$ the number
zeros of $f-a$ in the disk $\{z\colon |z|\leq r\}$. Here we ignore multiplicities; that is,
multiple zeros are counted only once.
(The notation $n(r,a,f)$ is used in Nevanlinna theory when multiplicities are counted.)
One may also take $a=\infty$, in which case we count the poles of $f$.

As usual in Nevanlinna theory, we put
\begin{equation}\label{cy}
\overline{N}(r,a,f) = \int_0^r \frac{\overline{n}(t,a,f)-\overline{n}(0,a,f)}{t}dt +\overline{n}(0,a,f)\log r
\end{equation}
and we denote by $S(r,f)$ any quantity that satisfies 
$S(r,f)=o(T(r,f))$ as $r\to\infty$, possibly outside some exceptional set of finite measure.

The following result~\cite[Theorem~2.5]{Hayman1964}
is a simple consequence of Nevanlinna's second fundamental theorem.
\begin{lemma}\label{lemma1}
Let $f$ be a meromorphic function and let $a_1,a_2,a_3$ be distinct small functions
(or constants in $\C\cup\{\infty\}$).  Then
\begin{equation}\label{cx}
T(r,f)\leq \sum_{j=1}^3 \overline{N}(r,a_j,f)+S(r,f).
\end{equation}
\end{lemma}
We remark that Yamanoi~\cite{Yamanoi} proved 
that if $\varepsilon>0$, $q\geq 3$ and $a_1,\dots,a_q$ are small functions, then
\begin{equation}\label{cz}
(q-2-\varepsilon) T(r,f)\leq \sum_{j=1}^3 \overline{N}(r,a_j,f)
\end{equation}
outside some exceptional set, but this result lies much deeper.

We shall need that if $j<k$, then $f_j$ is a small function with respect to~$f_k$; that is,
\begin{equation}\label{c0}
T(r,f_j)=o(T(r,f_k))
\quad\text{as}\ r\to\infty \ \text{if}\ j<k.
\end{equation}
Of course, this follows directly from Proposition~\ref{prop2},
but it  is also an immediate consequence of the the result~\cite[Lemma~ 2.6]{Hayman1964} that if 
$f$ and $g$ are transcendental entire functions, then
\begin{equation}\label{c1}
T(r,f)=o(T(r,f\circ g))
\quad\text{as}\ r\to\infty.
\end{equation}
Alternatively, we could use that
\begin{equation}\label{c1a}
T(r,g)=o(T(r,f\circ g))
\quad\text{as}\ r\to\infty.
\end{equation}
The latter result is an exercise in Hayman's book~\cite[p.~54]{Hayman1964}. For a thorough
discussion of these and related result we also refer to a paper by Clunie~\cite{Clunie1970}.
\begin{proof}[Proof of Proposition~\ref{prop1}]
We denote by $\overline{n}_A(r)$ the number of parameters $\lambda$ 
in $\{z\colon 0<|z|\leq r\}$ which satisfy~\eqref{b1}
and by $\overline{n}_B(r)$ the number of those $\lambda$ in $\{z\colon 0<|z|\leq r\}$
for which there exist $i,j\in\N$ satisfying $0<i<j<k+l$ and 
$E_\lambda^i(0)= E_\lambda^j(0)$; that is, $f_i(\lambda)=f_j(\lambda)$.
We also put
\begin{equation}\label{c2}
\overline{N}_A(r)=\int_0^r \frac{\overline{n}_A(t)}{t} dt
\quad\text{and}\quad
\overline{N}_B(r)=\int_0^r \frac{\overline{n}_B(t)}{t} dt.
\end{equation}
Then $n(r)=\overline{n}_A(r)-\overline{n}_B(r)$ and
\begin{equation}\label{c2a}
N(r)=\overline{N}_A(r)-\overline{N}_B(r).
\end{equation}

We apply Lemma~\ref{lemma1} with $f=f_{k+l}$, $a_1=0$, $a_2=f_k$ and $a_3=\infty$.
Note that the choice $a_2=f_k$ is admissible by~\eqref{c0}.
We have $\overline{N}(r,0,f_{k+l})=\log r$ and $\overline{N}(r,\infty,f_{k+l})=0$.
Noting that $\overline{N}(r,f_k,f_{k+l})$ and $\overline{N}_A(r)$ count the same points,
except that $0$ is counted in $\overline{N}(r,f_k,f_{k+l})$ but not in $\overline{N}_A(r)$,
we see that $\overline{N}(r,f_k,f_{k+l})=\overline{N}_A(r)+\log r$.
We thus deduce from  Lemma~\ref{lemma1} that
\begin{equation}\label{c3}
T(r,f_{k+l})\leq 
\overline{N}_A(r)+S(r,f_{k+l}).
\end{equation}
On the other hand, the first fundamental theorem of Nevanlinna theory and~\eqref{c0} imply that
\begin{equation}\label{c4}
\begin{aligned}
\overline{N}_A(r)
&=\overline{N}(r,f_k,f_{k+l})-\log r
\leq T(r,f_{k+l}-f_k)+O(1)
\\ &
\leq T(r,f_{k+l})+T(r,f_k)+O(1)
=(1+o(1))T(r,f_{k+l}).
\end{aligned}
\end{equation}
Combining the last two equations we find that
\begin{equation}\label{c5}
\overline{N}_A(r)=T(r,f_{k+l})+S(r,f_{k+l}).
\end{equation}
The first fundamental theorem also yields that
\begin{align}\label{c6}
\overline{N}_B(r)
&\leq \sum_{0<i<j<k+l} N(r,f_i,f_j)
\leq \sum_{0<i<j<k+l} T(r,f_j-f_i) +O(1)
\\ &
\leq \sum_{0<i<j<k+l} T(r,f_j)+T(r,f_i)+O(1)
= O\!\left(\sum_{0<j<k+l} T(r,f_j) \right)
\end{align}
so that
\begin{equation}\label{c7}
\overline{N}_B(r)=o(T(r,f_{k+l}))
\end{equation}
by~\eqref{c0}.
The conclusion now follows from~\eqref{c2a}, \eqref{c5} and~\eqref{c7}.
\end{proof}

\begin{remark}
The ideas used in the above proof are similar to those employed by 
Baker (see~\cite{Baker1960} or~\cite[Section~2.8]{Hayman1964}) in his proof 
that  a transcendental entire function has periodic points of period $p$ for 
all $p\in\N$, with at most one exception. 
His conjecture that $p=1$ is the only possible exception
was proved in~\cite{Bergweiler1991}.
\end{remark}

\section{Proof of Proposition~\ref{prop2}}  \label{sec3}
An exercise in Hayman's book~\cite[p.~7]{Hayman1964} is to show that
\begin{equation}\label{d0a}
T\!\left(r,e^{e^z}\right)\sim  \frac{e^r}{\sqrt{2\pi^3 r}} .
\end{equation}
The computations here are similar, but somewhat more involved.

The proof of Proposition~\ref{prop2} we give below is self-contained, but we note that using results of 
Hayman~\cite{Hayman1956} the proof can be shorted. More specifically, Lemmas~\ref{lemma6} and~\ref{lemma7}
below can be replaced by a reference to results of this paper; see the remark at the end of this section.

We define
\begin{equation}\label{a2}
a_k(r)=\frac{d\log f_k(r)}{d\log r}=\frac{rf'_k(r)}{f_k(r)}
\quad\text{and}\quad
b_k(r)=\frac{d \; a_k(r)}{d\log r}=ra'_k(r) .
\end{equation}
We also put
\begin{equation}\label{d0}
F_k(z)=\prod_{j=1}^k f_j(z),
\end{equation}
with $F_0(z)=1$.
\begin{lemma}\label{lemma8}
Let $k\geq 2$. Then
\begin{equation}\label{d3}
a_k(r)\sim F_{k-1}(r)
\quad\text{and}\quad
b_k(r)\sim F_{k-1}(r)F_{k-2}(r)=f_{k-1}(r)F_{k-2}(r)^2.
\end{equation}
\end{lemma}
\begin{proof}
Since $zf_k'(z)=f_k(z)+f_k(z)zf_{k-1}'(z)$ we see by induction that
\begin{equation}\label{h8}
zf_k'(z)=\sum_{m=0}^{k-1} \prod_{l=0}^{m} f_{k-l}(z)=
F_k(z)\sum_{j=0}^{k-1}\frac{1}{F_j(z)}.
\end{equation}
Hence
\begin{equation}\label{h10}
a_k(r)=F_{k-1}(r)\sum_{j=0}^{k-1}\frac{1}{F_j(r)} \sim F_{k-1}(r)
\end{equation}
as claimed.
The asymptotics for $b_k(r)$ 
follow from this by a  straightforward calculation.
\end{proof}

By $\log f_k$ we denote the branch of the logarithm which is real on the positive real axis.
\begin{lemma}\label{lemma6}
Let $k\geq 2$ and $r\geq 1$. Then
\begin{equation}\label{d1}
\log f_k(re^\tau)= \log f_k(r) +a_k(r)\tau + \frac12 b_k(r)\tau^2+ R(\tau)
\end{equation}
where
\begin{equation}\label{d2}
|R(\tau)|\leq 6\cdot 3^{3(k-1)}F_{k-1}(r)F_{k-2}(r)^2|\tau|^3
\quad\text{for}\ |\tau|\leq \frac{1}{2\cdot 3^{k-1} F_{k-2}(r)}.
\end{equation}
\end{lemma}
\begin{proof}
We first show by induction that if $j\in\N$ and $r\geq 1$, then
\begin{equation}\label{i1}
f_j(re^t)\leq (1+3^j F_{j-1}(r)t)f_j(r)\leq 2 f_j(r)
\quad\text{for}\ t\leq \frac{1}{3^jF_{j-1}(r)}.
\end{equation}
This is clear for $j=1$ in which case this just says that
\begin{equation}\label{i2}
re^t\leq (1+3t)r\leq 2r
\quad\text{for}\ t\leq \frac{1}{3}.
\end{equation}
Assuming that~\eqref{i1} holds, we find that if $t\leq 1/(3^{j+1}F_j(r))$ and $r\geq 1$, then
also $t\leq 1/(3^{j}F_{j-1}(r))$ and thus
\begin{equation}\label{i3}
\begin{aligned}
f_{j+1}(re^t)
& = r e^t \exp f_j(re^t)
%\\ &
\leq  r e^t \exp\!\left( (1+3^j F_{j-1}(r)t)f_j(r)\right)
\\ &
=  r e^t \exp\!\left( f_j(r)+3^j F_{j}(r)t\right)
%\\ &
=  f_{j+1}(r) \exp\!\left( (1+3^j F_{j}(r))t\right)
\\ &
\leq  f_{j+1}(r) \exp\!\left( 2\cdot 3^j F_{j}(r)t\right)
\leq  f_{j+1}(r) \left( 1+ 3^{j+1} F_{j}(r)t\right).
\end{aligned}
\end{equation}
This proves~\eqref{i1}.

We put
\begin{equation}\label{h1}
h(\tau)=\log f_k(re^\tau)= \log r +\tau +f_{k-1}(re^\tau).
\end{equation}
Noting that~\eqref{d1} is nothing else than the Taylor expansion of $h$ with remainder $R(\tau)$
we deduce that
(see, e.g., \cite[p.~126]{Ahlfors1966})
\begin{equation}\label{h2}
R(\tau)=\frac{\tau^3}{2\pi i} \int_{|w|=s}\frac{h(w)}{w^3(w-\tau)}dw
\end{equation}
if $s> |\tau|$. With $s=1/(3^{k-1}F_{k-2}(r))$ we find that if $|\tau|\leq s/2$, then
\begin{equation}\label{h3}
\begin{aligned}
|R(\tau)|
& \leq 
\frac{2|\tau|^3}{s^3}\max_{|w|=s}|h(w)|
%\\ &
\leq \frac{2|\tau|^3}{s^3}(\log r+s +f_{k-1}(re^s))
\\ &
\leq \frac{2|\tau|^3}{s^3}(\log r+s +2 f_{k-1}(r))
%\\ &
\leq \frac{6|\tau|^3}{s^3} f_{k-1}(r)
\\ &
= 6\cdot 3^{3(k-1)} F_{k-1}(r) F_{k-2}(r)^2 |\tau|^3.
\end{aligned}
\end{equation}
This is~\eqref{d2}. 
\end{proof}
We have restricted to $k\geq 2$ in Lemma~\ref{lemma6}, but we note that~\eqref{d1} trivially holds
for $k=1$ with $a_1(r)=1$, $b_1(r)=0$ and $R(\tau)=0$.

We will actually use Lemma~\ref{lemma6} not for the computation of $T(r,f_k)$, but for that of 
\begin{equation}\label{d4}
T(r,f_{k+1})=\frac{1}{2\pi}\int_{-\pi}^\pi \log^+|f_{k+1}(re^{i\theta})| d\theta.
\end{equation}
Here $\log^+ x=\max\{\log x,0\}$. The notation 
$h^+(x)=\max\{h(x),0\}$ will also be used for other functions $h$ in the sequel.

We will split the integral in~\eqref{d4} into two parts by considering the ranges $|\theta|\leq \delta(r)$
and $\delta(r)\leq |\theta|\leq\pi$ separately, for a suitably chosen function $\delta(r)$.
It will be convenient to chose 
\begin{equation}\label{d4a}
\delta(r)=\frac{1}{F_{k-1}(r)^{2/5}}.
\end{equation}
Then Lemma~\ref{lemma6} can be applied for $|\theta|\leq \delta(r)$, with an error term
$R(i\theta)$ satisfying $R(i\theta)=o(1)$.

To deal with the range $\delta(r)\leq |\theta|\leq\pi$ we will use the following lemma.
\begin{lemma}\label{lemma7}
If $k\geq 2$,  $\delta(r)\leq |\theta|\leq\pi$ and $r$ is sufficiently large, then
\begin{equation}\label{d5}
 \log |f_{k+1}(re^{i\theta})|\leq \frac{f_{k}(r)}{f_{k-1}(r)}.
\end{equation}
\end{lemma}
\begin{proof}
Put $g_1(\theta)=r\cos\theta$ and $g_{j}(\theta)=r\exp {g_{j-1}(\theta)}$ for $j\geq 2$.
Noting that $g_2(\theta)=re^{r\cos\theta}=|f_{2}(re^{i\theta})|$ and 
\begin{equation}\label{d5a}
|f_{j}(re^{i\theta})|=r\exp \re (f_{j-1}(re^{i\theta}))\leq r\exp |f_{j-1}(re^{i\theta})|
\end{equation}
for $j\geq 3$ we see by induction that
\begin{equation}\label{k1}
|f_{j}(re^{i\theta})|\leq g_j(\theta)
\end{equation}
for all $j\geq 2$.

Since $\cos\theta\leq 1-\theta^2/4$ for $|\theta|\leq 1$ we have
\begin{equation}\label{k2}
\begin{aligned}
g_2(\theta)
&=
re^{r\cos\theta}
\leq re^r \exp\!\left(-r\frac{\theta^2}{4}\right)
\\ &
= f_2(r) \exp\!\left(-\frac{F_1(r)}{4}\theta^2\right)
\quad\text{for}\ |\theta|\leq  1.
\end{aligned}
\end{equation}
We shall show by induction that if $j\geq 2$ and $r\geq 1$, then
\begin{equation}\label{k3}
g_j(\theta)
\leq  f_j(r) \exp\!\left(-\frac{F_{j-1}(r)}{2^j}\theta^2\right)
\quad\text{for}\ |\theta|\leq \frac{1}{\sqrt{F_{j-2}(r)}}.
\end{equation}
Note that~\eqref{k2} says that this holds for $j=2$. Suppose now that $j\geq 2$
and that~\eqref{k3} holds. Let $|\theta|\leq 1/\sqrt{F_{j-1}(r)}$.
Then $|\theta|\leq 1/\sqrt{F_{j-2}(r)}$ since $r\geq 1$.
Noting that $e^{-x}\leq 1-x/2$ for $0\leq x\leq 1$ we obtain
\begin{equation}\label{k4}
\begin{aligned}
g_{j+1}(\theta)
& = r\exp g_j(\theta)
%\\ &
\leq r \exp\!\left(  f_j(r) \exp\!\left(-\frac{F_{j-1}(r)}{2^j}\theta^2\right)\right)
\\ &
\leq r \exp\!\left(  f_j(r) \left(1-\frac{F_{j-1}(r)}{2^{j+1}}\theta^2\right)\right)
%\\ &
=  f_{j+1}(r) \exp\!\left(-\frac{F_{j}(r)}{2^{j+1}}\theta^2\right).
\end{aligned}
\end{equation}
Hence~\eqref{k3} holds for all $j\geq 2$.

Suppose now that $\delta(r)\leq |\theta|\leq\pi$. Then
\begin{equation}\label{k5}
 \log |f_{k+1}(re^{i\theta})|\leq  \log g_{k+1}(\theta) \leq  \log g_{k+1}(\delta(r))
 = g_{k}(\delta(r))  + \log r
\end{equation}
by~\eqref{k1}.
Since $\delta(r)=1/F_{k-1}(r)^{2/5}\leq 1/\sqrt{F_{k-2}(r)}$ 
for large $r$ we deduce from the last inequality and~\eqref{k3} that
\begin{equation}\label{k6}
\begin{aligned}
 \log |f_{k+1}(re^{i\theta})|
& \leq
f_{k}(r) \exp\!\left(-\frac{F_{k-1}(r)}{2^{k}}\delta(r)^2\right)
+\log r
\\ &
= 
f_{k}(r) \exp\!\left(-\frac{F_{k-1}(r)^{1/5}}{2^{k}}\right) 
+\log r
\leq \frac{f_{k}(r)}{f_{k-1}(r)},
\end{aligned}
\end{equation}
if $r$ is sufficiently large.
\end{proof}
\begin{lemma} \label{ec}
\begin{equation}\label{lec}
\lim_{t\to\infty}\int_{-\infty}^\infty e^{-x^2}\cos^+(tx)dx=\frac{1}{\sqrt{\pi}}.
\end{equation}
\end{lemma}
\begin{proof}
Integration by parts yields
\begin{equation}\label{lec1}
\int_{-\infty}^\infty e^{-x^2}\cos^+(tx)dx=
\int_{-\infty}^\infty e^{-x^2} 2x\int_0^x \cos^+(ty)dy\; dx.
\end{equation}
Since
\begin{equation}\label{lec2}
\int_0^x \cos^+(ty)dy \sim \frac{x}{\pi}
\quad\text{as}\ t\to\infty,
\end{equation}
locally uniformly in $\R\backslash\{0\}$, we obtain
\begin{equation}
\label{lec3}
\lim_{t\to\infty}\int_{-\infty}^\infty e^{-x^2}\cos^+(tx)dx=
\frac{2}{\pi}\int_{-\infty}^\infty e^{-x^2} x^2 dx=\frac{1}{\sqrt{\pi}}
\end{equation}
as claimed.
\end{proof}

\begin{proof}[Proof of Proposition~\ref{prop2}]
It follows from Lemma~\ref{lemma6} that
\begin{equation}\label{s4}
f_k(re^{i\theta})= f_k(r)\exp\!\left(ia_k(r)\theta-\tfrac12 b_k(r)\theta^2\right)(1+S(\theta))
\quad\text{for}\ |\theta|\leq \delta(r),
\end{equation}
where
\begin{equation}\label{s5}
\begin{aligned}
|S(\theta)|
&
= \left| e^{R(i\theta)}-1 \right|
\leq 2 |R(i\theta)|
\\ &
\leq 12\cdot 3^{3(k-1)} F_{k-1}(r) F_{k-2}(r)^2 \delta(r)^3 
= 12\cdot 3^{3(k-1)} \frac{F_{k-2}(r)^2}{F_{k-1}(r)^{1/5}}
\end{aligned}
\end{equation}
for large $r$ and hence $S(\theta)=o(1)$ as $r\to\infty$. This implies that
\begin{equation}\label{q5a}
\begin{aligned}
\re\!\left( f_k(re^{i\theta} \right)
= f_k(r)\exp\!\left(-\tfrac12 b_k(r)\theta^2\right)\cos(a_k(r)\theta)+o\!\left( f_k(r)\exp\!\left(-\tfrac12 b_k(r)\theta^2\right)\right)
\end{aligned}
\end{equation}
and thus
\begin{equation}\label{q6}
\re^+\!\left( f_k(re^{i\theta} \right)
= f_k(r)\exp\!\left(-\tfrac12 b_k(r)\theta^2\right)\left(\cos^+(a_k(r)\theta)+o(1)\right)
\quad\text{for}\ |\theta|\leq \delta(r),
\end{equation}
where the term $o(1)$ is uniform in $\theta$.

We conclude that
\begin{equation}\label{q7}
\begin{aligned}
&\quad \ \int_{-\delta(r)}^{\delta(r)}
\log^+|f_{k+1}(re^{i\theta})| d\theta
\\
&=
f_k(r) \int_{-\delta(r)}^{\delta(r)}
\exp\!\left(-\tfrac12 b_k(r)\theta^2\right)\left(\cos^+(a_k(r)\theta)+o(1)\right)d\theta.
\\ &
=\frac{\sqrt{2}f_k(r)}{\sqrt{b_k(r)}}
 \int_{-c(r)}^{c(r)}
\exp\!\left(-u^2\right)\left(\cos^+\!\left(\frac{\sqrt{2}a_k(r)}{\sqrt{b_k(r)}}u\right)+o(1)\right)\!du
\end{aligned}
\end{equation}
with
\begin{equation}\label{q8}
c(r)=\frac{\sqrt{b_k(r)}\delta(r)}{\sqrt{2}}
=(1+o(1))\frac{F_{k-1}(r)^{1/10}\sqrt{F_{k-2}(r)}}{\sqrt{2}}
\to\infty
\end{equation}
by Lemma~\ref{lemma8}. The same lemma yields that
\begin{equation}\label{q8a}
\frac{a_k(r)}{\sqrt{b_k(r)}}
=(1+o(1))\frac{\sqrt{F_{k-1}(r)}}{\sqrt{F_{k-2}(r)}}
=(1+o(1))\sqrt{f_{k-1}(r)}
\to\infty.
\end{equation}
Lemma~\ref{lec} now implies that
\begin{equation}\label{q9}
\int_{-\delta(r)}^{\delta(r)} \log^+|f_{k+1}(re^{i\theta})| d\theta
\sim 
\frac{\sqrt{2}f_k(r)}{\sqrt{\pi b_k(r)}}.
\end{equation}
Since
\begin{equation}\label{q10}
\log^+|f_{k+1}(re^{i\theta})|\leq \frac{f_k(r)}{f_{k-1}(r)}=o\!\left( \frac{f_k(r)}{\sqrt{b_k(r)}}\right)
\quad\text{for}\ \delta(r)\leq |\theta|\leq \pi
\end{equation}
by Lemma~\ref{lemma7} and Lemma~\ref{lemma8} we conclude that
\begin{equation}\label{q11}
\int_{-\pi}^{\pi} \log^+|f_{k+1}(re^{i\theta})| d\theta
\sim 
\frac{\sqrt{2}f_k(r)}{\sqrt{\pi b_k(r)}}.
\end{equation}
Thus
\begin{equation}\label{q12}
\begin{aligned}
T(r,f_{k+1})
&=\frac{1}{2\pi}\int_{-\pi}^{\pi} \log^+|f_{k+1}(re^{i\theta})| d\theta
\\ &
\sim 
\frac{f_k(r)}{\sqrt{2\pi^3 b_k(r)}}
\sim 
\frac{f_k(r)}{\sqrt{2\pi^3}\sqrt{f_{k-1}(r)}F_{k-2}(r)}
\end{aligned}
\end{equation}
by Lemma~\ref{lemma8}. The conclusion follows with $k=m-1$.
\end{proof}

\begin{remark}
An entire function $f$ is called \emph{admissible} in the sense of Hayman~\cite{Hayman1956} 
if $f(r)=M(r,f)$ for large $r$ and if with
\begin{equation}\label{q2}
a(r)=\frac{d \log M(r,f)}{d\log r}=\frac{rf'(r)}{f(r)}
\quad\text{and}\quad
b(r)=\frac{d \; a(r)}{d\log r}=ra'(r)
\end{equation}
there exists $\delta(r)\in (0,\pi]$ such that, as $r\to\infty$,
\begin{equation}\label{a3}
f(re^{i\theta})\sim f(r)\exp\!\left(ia(r)\theta-\tfrac12 b(r)\theta^2\right)
\quad\text{for}\ |\theta|\leq \delta(r)
\end{equation}
and
\begin{equation}\label{a4}
f(re^{i\theta})= \frac{o(f(r))}{\sqrt{b(r)}}
\quad\text{for}\ \delta(r)\leq |\theta|\leq\pi.
\end{equation}
Moreover, it is assumed that $b(r)\to\infty$ as $r\to\infty$.

Hayman~\cite[Theorems VI and VIII]{Hayman1956} showed that
if $f$ is admissible, then so are $e^f$ and $fP$ for any real polynomial $P$ with
positive leading coefficient. This implies that $f_k$ is admissible for $k\geq 2$.

The admissibility of $f_k$ immediately yields slightly weaker versions of Lemmas~\ref{lemma6} and~\ref{lemma7},
but these versions are strong enough to prove Proposition~\ref{prop2}.
In fact, the arguments used in the above proof yield the following Proposition~\ref{prop3}.
Since its proof is largely analogous to that of Proposition~\ref{prop2}, replacing 
Lemmas~\ref{lemma6} and~\ref{lemma7} by a reference to~\eqref{a3} and~\eqref{a4}, 
we will only sketch the proof.
\end{remark}
\begin{proposition}\label{prop3}
Let $f$ be an admissible entire function and let $b(r)$ be defined by~\eqref{q2}. Then
\begin{equation}\label{q5}
T(r,e^f) \sim \frac{1}{\sqrt{2\pi^3}}
\frac{f(r)}{\sqrt{b(r)}}
\end{equation}
\end{proposition}
\begin{proof}[Sketch of proof]
First we note that~\eqref{a3} means that~\eqref{s4} holds
with $f_k$ replaced by $f$ and $S(\theta)=o(1)$ for $|\theta|\leq \delta(r)$.
We proceed as in the proof of Proposition~\ref{prop2}. 
To see that $c(r)=\delta(r)\sqrt{b(r)/2}\to\infty$ as 
in~\eqref{q8} we note that we may choose $\theta=\delta(r)$ in both~\eqref{a3} and~\eqref{a4}.
This yields
\begin{equation}\label{q10b}
f(r)\exp\!\left(-\tfrac12 b(r)\delta(r)^2\right)
=o\!\left( \frac{f(r)}{\sqrt{b(r)}}\right)
\end{equation}
and hence $\exp\!\left(-\tfrac12 b(r)\delta(r)^2\right)=o(1)$, from which we deduce that $c(r)\to\infty$.
We conclude that~\eqref{q9} holds with $f_k$ replaced by $f$ and $f_{k+1}$ replaced by~$e^f$; that is,
\begin{equation}\label{q9a}
\int_{-\delta(r)}^{\delta(r)} \log^+|e^{f(re^{i\theta})}| d\theta
\sim 
\frac{\sqrt{2}f(r)}{\sqrt{\pi b(r)}}.
\end{equation}
Moreover, 
\begin{equation}\label{q10a}
\log^+|e^{f(re^{i\theta})}|\leq |f(re^{i\theta})|=o\!\left( \frac{f(r)}{\sqrt{b(r)}}\right)
\quad\text{for}\ \delta(r)\leq |\theta|\leq \pi
\end{equation}
by~\eqref{a4}.
The conclusion follows directly from the last two equations.
\end{proof}
We note that Proposition~\ref{prop2} is an immediate consequence of Proposition~\ref{prop3}.

\section{Proof of the theorem}  \label{sec4}
A classical growth lemma of Borel (see~\cite[p.~90]{GO} or~\cite[Lemma~2.4]{Hayman1964}) says that if $\phi\colon [r_0,\infty)\to (0,\infty)$
is a continuous, increasing function, then there exists a subset $E$ of $[r_0,\infty)$ of finite measure 
such that
\begin{equation}\label{e0}
\phi\!\left(1+\frac{1}{\phi(r)}\right)\leq 2\phi(r) 
\quad\text{for}\ r\notin E .
\end{equation}
The exceptional set in Nevanlinna's second fundamental theorem and thus
the exceptional set $E$ in Proposition~\ref{prop1} arise from the application of this lemma
to the Nevanlinna characteristic.

If the function $\phi$ is sufficiently ``regular'', then
the inequality in Borel's lemma holds for all large~$r$.
In fact, boundedness of the exceptional set $E$ in Borel's lemma is sometimes
taken as a regularity condition; see, e.g.,~\cite[p.~245]{EdreiFuchs1964}.
The following lemma gives a simple condition implying that the exceptional set in this 
lemma is bounded. While I believe that this or similar results are well-known to the 
experts, I have not found this lemma in the literature.
\begin{lemma}\label{lemma3}
Let $\phi\colon [r_0,\infty)\to (0,\infty)$ be a non-decreasing, differentiable function satisfying
$\phi'(r)\leq \phi(r)^{3/2}$ for all~$r$. Then
\begin{equation}\label{e1}
\phi\!\left(1+\frac{1}{\phi(r)}\right)\sim\phi(r) 
\quad\text{as}\ r\to\infty.
\end{equation}
\end{lemma}
\begin{proof}
The result is trivial if $\lim_{r\to\infty} \phi(r)<\infty$.
We may thus assume that  $\lim_{r\to\infty} \phi(r)=\infty$.
For $r\geq r_0$ we have
\begin{equation}\label{e2}
\frac{1}{\sqrt{\phi(r)}} -
\frac{1}{\sqrt{\phi(r+1/\phi(r))}} 
=\frac12 \int_r^{r+1/\phi(r)} \frac{\phi'(t)}{\phi(t)^{3/2}} dt
\leq 
\frac{1}{2\phi(r)}
\end{equation}
and thus
\begin{equation}\label{e3}
\sqrt{\frac{\phi(r)}{\phi(r+1/\phi(r))}}\geq 1-\frac{1}{2\sqrt{\phi(r)}},
\end{equation}
from which the conclusion follows.
\end{proof}
A straightforward calculation shows that the right hand side of~\eqref{b6} satisfies the
hypothesis -- and thus the conclusion -- of Lemma~\ref{lemma3}.
From this it is not difficult to deduce that the exceptional set in Nevanlinna's second
fundamental theorem and in Lemma~\ref{lemma1} is bounded for $f=f_m$.
This implies that no exceptional set $E$ is required in Proposition~\ref{prop1}. Combining this with Proposition~\ref{prop2}
we find that under the hypotheses of Proposition~\ref{prop1} we have
\begin{equation}\label{e4}
N(r)\sim T(r,f_{k+l})
\sim \frac{1}{\sqrt{2\pi^3}} \frac{f_{k+l-1}(r)}{\sqrt{f_{k+l-2}(r)}F_{k+l-3}(r)}
\quad\text{as}\ r\to\infty,
\end{equation}
with $F_{k+l-3}(r)$ defined by~\eqref{d0}.

To obtain a result for $n(r)$ we use the following result of London~\cite[p.~502]{London}.
\begin{lemma}\label{lemma4}
Let
$\phi,\psi\colon [x_0,\infty)\to (0,\infty)$ be functions satisfying
\begin{equation}\label{e5}
\phi(x)\sim\psi(x)
\quad\text{as}\ x\to\infty.
\end{equation}
Suppose that $\psi$ is convex and that
$\phi$ is twice continuously differentiable, with $\phi'$ and $\phi''$ positive and $\phi'$ unbounded.
Suppose also that there exists a constant $\beta$ such that 
\begin{equation}\label{e5a}
\frac{\phi''(x)\phi(x)}{\phi'(x)^2}\leq \beta 
\end{equation}
for all $x\geq x_0$. Then
\begin{equation}\label{e6}
\phi'(x)\sim\psi'(x) \quad\text{as}\ x\to\infty.
\end{equation}
Here $\psi'$ denotes either the left or the right derivative of $\psi$
on the countable set for which these may be different.
\end{lemma}
Note that l'Hospital's rule says that~\eqref{e6} implies~\eqref{e5}.
Lemma~\ref{lemma4} may be considered as a reversal of l'Hospital's rule.
For this an additional hypothesis such as~\eqref{e5a} is essential.

\begin{proof}[Proof of the theorem]
We denote the right hand side of~\eqref{e4} by $g(r)$.
Since $N(r)$ is convex in $\log r$ we see that $\psi(x)=N(e^x)$ is convex in~$x$.
It is easy to see that $\phi(x)=g(e^x)$ satisfies the hypothesis of Lemma~\ref{lemma4}.
In fact, it is not difficult to see that $\phi''(x)\phi(x)/\phi'(x)^2\to 1$ as $x\to\infty$.
We thus deduce from Lemma~\ref{lemma4} that $\phi'(x)\sim\psi'(x)$ and hence 
that $n(r)\sim r g'(r)$.
From this the conclusion follows easily using Lemma~\ref{lemma8}.
\end{proof}

\medskip

\noindent
Mathematisches Seminar \\
Christian--Albrechts--Universit\"at zu Kiel \\
Ludewig--Meyn--Stra{\ss}e 4 \\
24098 Kiel \\
Germany

\medskip

\noindent
E-mail address: {\tt bergweiler@math.uni-kiel.de}
\end{document}